\documentclass[a4paper]{article}
\usepackage{hyperref}
\usepackage{url}

\usepackage{graphicx} 
\usepackage[round]{natbib}
\usepackage{algorithm}
\usepackage{algorithmic}

\usepackage{amsmath}
\usepackage{amsfonts}
\usepackage{amsthm}
\usepackage{bm}
\usepackage{nicefrac}
\usepackage{enumerate}
\allowdisplaybreaks

\theoremstyle{plain}
\newtheorem{thm}{\protect\theoremname}
\theoremstyle{plain}

\theoremstyle{definition}

\theoremstyle{plain}

\theoremstyle{plain}

\theoremstyle{plain}

\theoremstyle{plain}
\newtheorem{lem}[]{\protect\lemmaname}
\theoremstyle{definition}

\providecommand{\algorithmname}{Algorithm}
\providecommand{\assumptionname}{Assumption}
\providecommand{\corollaryname}{Corollary}
\providecommand{\definitionname}{Definition}
\providecommand{\propositionname}{Proposition}
\providecommand{\theoremname}{Theorem}
\providecommand{\remarkname}{Remark}
\providecommand{\lemmaname}{Lemma}

\begin{document}

\title{Non-injectivity of Bures--Wasserstein barycentres in infinite dimensions}

\author{Yoav Zemel \\
Institut de Math\'ematiques \\
\'Ecole polytechnique f\'ed\'erale de Lausanne
\\
\texttt{yoav.zemel@epfl.ch}
}

\maketitle

\vspace{-0.02cm}

\begin{abstract}
We construct a counterexample to the injectivity conjecture of \cite{masarotto2018procrustes}.  Namely, we construct a class of examples of injective covariance operators on an infinite-dimensional separable Hilbert space for which the Bures--Wasserstein barycentre is highly non injective --- it has a kernel of infinite dimension.
\end{abstract}

\global\long\def\CC{\mathbb{C}}
\global\long\def\SS{S^{1}}
\global\long\def\R{\mathbb{R}}
\global\long\def\actson{\curvearrowright}
\global\long\def\ra{\rightarrow}
\global\long\def\z{\mathbf{z}}
\global\long\def\ZZ{\mathbb{Z}}
\global\long\def\NN{\mathbb{N}}
\global\long\def\sgn{\mathrm{sgn}\:}
\global\long\def\RRpos{\RR_{>0}}
\global\long\def\var{\mathrm{var}}
\global\long\def\circint{\int_{-\pi}^{\pi}}
\global\long\def\F{\mathcal{F}}
\global\long\def\pb#1{\langle#1\rangle}
\global\long\def\op{\mathrm{op}}
\global\long\def\Op{\mathrm{op}}
\global\long\def\supp{\mathrm{supp}}
\global\long\def\ceil#1{\lceil#1\rceil}
\global\long\def\TV{\mathrm{TV}}
\global\long\def\floor#1{\lfloor#1\rfloor}
\global\long\def\vt{\vartheta}
\global\long\def\vp{\varphi}
\global\long\def\class#1{[#1]}
\global\long\def\of{(\cdot)}
\global\long\def\one{\mathbbm{1}}
\global\long\def\cov{\mathrm{cov}}
\global\long\def\CC{\mathbb{C}}
\global\long\def\SS{S^{1}}
\global\long\def\RR{\mathbb{R}}
\global\long\def\actson{\curvearrowright}
\global\long\def\ra{\rightarrow}
\global\long\def\z{\mathbf{z}}
\global\long\def\ZZ{\mathbb{Z}}
\global\long\def\h{\mu}
\global\long\def\convr{*_{\RR}}
\global\long\def\x{\mathbf{x}}
\global\long\def\ve{\epsilon}
\global\long\def\cv{\mathfrak{c}}
\global\long\def\wh#1{\hat{#1}}
\global\long\def\norm#1{\left|\left|#1\right|\right|}
\global\long\def\Tmean{T}
\global\long\def\Tslope{m}
\global\long\def\degC#1{#1^{\circ}\mathrm{C}}
\global\long\def\X{\mathbf{X}}
\global\long\def\bbeta{\boldsymbol{\beta}}
\global\long\def\b{\mathbf{b}}
\global\long\def\Y{\mathbf{Y}}
\global\long\def\H{\mathbf{H}}
\global\long\def\e{\bm{\epsilon}}
\global\long\def\s{\mathbf{s}}
\global\long\def\t{\mathbf{t}}
\global\long\def\dAc{\partial A_c}
\global\long\def\cl{\mathrm{cl}}
\global\long\def\Hoe{\text{H\"older}}
\global\long\def\bb#1{\mathbb{#1}}
\global\long\def\bX{\bm{X}}
\global\long\def\bY{\bm{Y}}
\global\long\def\D{\mathfrak{D}}
\global\long\def\dGH{d_{\mathcal{GH}}}
\global\long\def\dWp{d_{\mathcal{W},p}}
\global\long\def\umu{\underline{\mu}}
\global\long\def\FLB{\mathbf{FLB}}
\global\long\def\dX{\mathbf{d_\bX}}
\global\long\def\dY{\mathbf{d_\bY}}
\global\long\def\mX{{\mu_\bX}}
\global\long\def\mY{{\mu_\bY}}	
\global\long\def\cal#1{\mathbf{\mathcal{#1}}}
\global\long\def\UXn{U_{\bX n}}
\global\long\def\dXX{{\Delta_{\bX}}}
\global\long\def\dYY{{\Delta_{\bY}}}
\global\long\def\d{\:\mathrm{d}}
\global\long\def\Ds{\mathfrak{d}}
\global\long\def\ph{\hat{p}}
\global\long\def\muh{\hat{\mu}}
\global\long\def\onevec{\one}
\global\long\def\T{\mathcal{T}}
\global\long\def\root{\mathrm{root}}
\global\long\def\parent{\mathrm{par}}
\global\long\def\children{\mathrm{children}}
\global\long\def\scp#1#2{\langle #1, #2 \rangle}
\global\long\def\X{\mathcal{X}}
\global\long\def\br{\bm{r}}
\global\long\def\bs{\bm{s}}
\global\long\def\brh{\hat{\bm{r}}}
\global\long\def\bsh{\hat{\bm{s}}}
\global\long\def\rh{\hat{r}}
\global\long\def\sh{\hat{s}}
\global\long\def\bu{{\bm{u}}}
\global\long\def\bv{\bm{v}}
\global\long\def\bh{\bm{h}}
\global\long\def\bw{\bm{w}}
\global\long\def\bD{{\bm{D}}}
\global\long\def\bx{\bm{x}}
\global\long\def\by{\bm{y}}
\global\long\def\bZ{\bm{Z}}
\newcommand{\WhnmB}{\hat{W}^{(S)}}
\global\long\def\smax{{\sigma_{\mathrm{max}}}}
\global\long\def\sT{\sigma_\T}
\global\long\def\bsT{\bm{\sigma}_\T}
\global\long\def\E{\mathcal{E}}
\global\long\def\V{\mathcal{V}}
\global\long\def\cmax{c_{\mathrm{max}}}
\global\long\def\h{h}
\global\long\def\Id{\mathrm{Id}}
\global\long\def\N{\mathcal{N}}
\global\long\def\diam{\:\mathrm{diam}}
\global\long\def\ns{{n^{*}}}
\global\long\def\ps{{p^{*}}}
\global\long\def\Wt{W^{(t)}}
\global\long\def\BL1{\mathrm{BL}_1}
\global\long\def\I{\mathcal{I}}
\global\long\def\simD{\stackrel{D}{\sim}}
\global\long\def\d{\mathrm{d}}
\global\long\def\dpp{d^p\!}
\newcommand{\innprod}[2]{{\left\langle {#1},{#2}\right\rangle}}

\section{Introduction}
Wasserstein barycenters are Fr\'echet means with respect to the Wasserstein distance of optimal transport (see e.g., the textbooks \cite{rachev1998mass}, \cite{villani2003topics},  \cite{santambrogio2015optimal}, and \cite{panaretos2020invitation}). Introduced by \cite{agueh2011barycenters}, Wasserstein barycentres have been shown to define meaningful notions of average for objects with complex geometric structures.  As such, they have found applications in numerous fields.  Examples include unsupervised dictionary learning \citep{schmitz2018wasserstein}, distributional clustering \citep{ye2017fast}, Wasserstein principal component analysis \citep{seguy2015principal}, neuroimaging \citep{gramfort2015fast} and computer vision \citep{rabin2011wasserstein,solomon2015convolutional,bonneel2016wasserstein}.

The first regularity property for Wasserstein barycenters has been already shown in \cite{agueh2011barycenters}, and states that if at least one of a finite collection of probability measures $\mu_1,\dots,\mu_n$ (with finite second moment) on $\R^d$ is absolutely continuous with a bounded density, then their barycentre $\overline \mu$ is also absolutely continuous with bounded density.  It therefore follows that the optimal transport from the barycentre $\overline\mu$ to each of the original measures $\mu_i$ is given by a transport map, even if $n-1$ of the measures possess no regularity. Perhaps surprisingly, the existence of the transport maps also holds when all the measures are finitely supported on $\R^d$ (\cite{anderes2016discrete}; see also \cite{borgwardt2017strongly}).

When the measures $\mu_1,\dots,\mu_n$ are all Gaussian, the Wasserstein distance has a closed form \citep{olkin1982distance}.  While the barycentre (known as the Bures--Wasserstein barycentre in the Gaussian case) does not admit one in the general case, it is known to be Gaussian itself \citep{agueh2011barycenters}.  Since its behaviour with respect to the mean vectors of the Gaussian distributions is trivial, it is commonplace in the Gaussian case to identify the measures with the corresponding covariances matrices (after setting all the mean vectors to be zero), so that the Bures--Wasserstein barycentre can be viewed as a barycentre of the covariance matrices.   The finite-dimensional Gaussian case of Wasserstein barycentres was already discussed by \cite{agueh2011barycenters}, and studied in more detail by a number of authors, including \cite{alvarez2016fixed} \cite{bhatia}, \cite{zemel2019frechet}, \cite{chewi2020gradient}, \cite{kroshnin2021statistical}, \cite{carlier2021entropic}, and others.

\cite{cuesta1996lower} showed that the results of \cite{olkin1982distance} are also valid when $\R^d$ is replaced by an infinite dimensional separable Hilbert space, where the analogue of covariance matrix is known as covariance operator.  \cite{masarotto2018procrustes} then studied the Bures--Wasserstein barycentre problem on separable Hilbert spaces, in the process identifying it the Procrustes distance on covariance operators considered by \cite{pigoli2014distances}.  An important question raised by \cite{masarotto2018procrustes} is whether transport maps from the Bures--Wasserstein barycentre to each of the covariance must exist, as they do in finite dimensions.  This question is not merely a mathematical curiosity; these maps are closely related to the log maps that lift measures to the tangent space \cite{ambrosio2008gradient}.  As highlighted by \cite{zemel2019frechet}, \cite{masarotto2018procrustes}, and others, this allows to apply linear techniques such as principal component analysis to the nonlinear space of Gaussian measures (or of covariance operators).  The approach employed by \cite{masarotto2018procrustes} was based on the fact, establised by \cite{cuesta1996lower}, that a transport map from $N(0,\Sigma_1)$ to $N(0,\Sigma_2)$ exists provided $\ker(\Sigma_1)\subseteq\ker(\Sigma_2)$, i.e., $\Sigma_1$ is ``more injective" than $\ker(\Sigma_2)$.  Thus \cite{masarotto2018procrustes} were led to conjecture that if all $\Sigma_1,\dots,\Sigma_n$ are all injective, then so is the barycentre $\overline \Sigma$ (see their Conjecture 17).

The injectivity conjecture is true, of course, in finite dimensions, and it also holds in infinite dimensions when the covariance operators commute ($\Sigma_i\Sigma_j=\Sigma_j\Sigma_i$).  Though they have not proven the conjecture, \cite{masarotto2022transportation} have been able to establish the existence of transport maps as bounded operators, thus guaranteeing the availability of log maps and the validity of a principal component analysis procedure.  Nevertheless, the injectivity conjecture remains open.  While no longer needed in order to employ principal component analysis, the conjecture is still relevant for large sample theory of Bures--Wasserstein barycentres; indeed, it is related to a condition recently assumed by \cite{santoro2023large} in their study of a central limit theorem for empirical Bures--Wasserstein barycenters in infinite dimensions.

This paper resloves the injectivity condition in an unequivocally negative manner.  We construct generic counterexamples that show how the conjecture is blatantly false.  The construction is, in our view, striking in its simplicity, and is all the more surprising given the existence of transport maps established by \cite{masarotto2022transportation}.  In short, we show that any sufficiently nondegenerate covariance $\Sigma$ can be the barycentre of injective covariances, even when $\Sigma$ itself is very far from being injective.  Interestingly, it is the very existence of transport maps that provides a clear path into constructing the counterexamples.

\textbf{Notation}. Below, $H$ is an infinite-dimensional separable Hilbert space, and $\Sigma$ (possibly with subscript) is a covariance operator on $H$, that is, a self-adjoint nonnegative definite trace-class linear operator from $H$ to $H$.


\section{Negative resolution of the injectivity conjecture}
\begin{thm}
Let $\Sigma$ be a covariance operator with infinite-dimensional image.  Then for $n\ge2$ there exists a collection $S_1,\dots,S_n$ of injective covariance operators for which $\Sigma$ is the (unique) barycentre.  A population version (with $S_1,\dots,S_n$ replaced by a random operator $S$) is also possible; see the end of the proof.
\end{thm}
Obviously the conclusion holds if the kernel of $\Sigma$ has a finite dimension.  This theorem is as strong as it could be, because we know that a bounded linear optimal transport map $T_i$ from $\Sigma$ to an injective operator $S_i$ exists, so that $S_i=T_i\Sigma T_i$ needs to be injective, and with $\rm{dim}(H)=\infty$ this is impossible when the range of $\Sigma$ has finite rank.

The gist of the construction is to ``create injectivity" by exhibiting a nice $T$ such $T\Sigma T$ is injective even if $\Sigma$ is not.  It is noteworthy, however, that the existence of $T$ (established in \cite{masarotto2022transportation}) is not used in our proof.
\begin{lem}
For $\Sigma$ as above, it is possible to find a bounded self-adjoint nonnegative operator $T$ such that $T\Sigma T$ is injective.
\end{lem}
\begin{proof}
It is possible to find an orthonormal basis $(\varphi_k)_{k=1}^\infty$ of $H$ such that the kernel of $\Sigma$ is contained in the span of $\varphi_1,\varphi_3,\varphi_5,\dots$.

Since $T\Sigma T$ has to be injective, $T$ must be injective, but at the same time its range must be disjoint from the kernel of $\Sigma$, the span of the odd basis functions.  It is easy to construct such an operator:  define $F:H\to H$ by $F\varphi_k=\varphi_{2k}$.  Now $F$ is not self-adjoint, but $F+F^*$ is.  This operator is not nonnegative definite, but by writing $x=\sum x_k\varphi_k$ with $\sum x_k^2=1$ and observing that
\[
\innprod {F^*x}x=
\innprod {Fx}x
=\innprod{\sum x_k\varphi_{2k}}{\sum x_i\varphi_i}
=\sum_{k=1}^\infty x_kx_{2k}
\]
is in $[-1,1]$ by the Cauchy--Schwarz inequality, we obtain that $\|F\|_\infty\le1$ and the same holds for $\|F^*\|_\infty$, where $\|\cdot\|_\infty$ denotes the operator norm.  It therefore follows that the operator\footnote{Multiplying the identity by any $c\ge2$ will do.  The relevant polynomial $1\pm ct+t^2$ has a root of magnitude $\ge1$ for any $c\in\R$, and $T$ is nonnegative definite for any $c\ge2$.}
\[
T=F+F^*+2\mathcal I
\]
is self-adjoint, nonnegative, and with $\|T\|_\infty\le 4$.

We now show that $T$ is our desired operator.  If a vector $x\in H$ satisfies $Tx\in \mathrm{ker}(\Sigma)$ then $\innprod{Tx}{\varphi_{k}}=0$ for all $k$ even.  We shall show that $x=0$.  This would imply that $\Sigma T$ is injective, and therefore $T$ is injective, so that $T\circ (\Sigma T)$ is injective, as required.

Expressing $x=\sum x_k\varphi_k$ and noting that $F^*\varphi_k=\varphi_{k/2}$ (zero if $k$ is odd), we obtain the sequence of equalities
\[
0=\innprod{Tx}{\varphi_k}
=x_{2k}+x_{k/2}+2x_k,\qquad k=2,4,6,\dots,
\]
so that
\[
x_{4k} = -2x_{2k} - x_{k},\qquad k=1,2,3,\dots.
\]
If $x\ne 0$ this implies that $x_k$ cannot go to zero, as the relevant polynomial $1+2t+t^2$ has a root of magnitude $\ge1$.  More precisely fix $k\ge1$ and let $y_j=x_{2^jk}$ so that $x_k=y_0$ and
\[
y_j = -2y_{j-1} - y_{j-2},\qquad j=2,3,\dots.
\]
Define the generating function
\[
g_k(t)=\sum_{j=0}^\infty y_jt^j = y_0 +  y_1t - \sum_{j=2}^\infty 2y_{j-1} t^{j-1}t - \sum_{j=2}^\infty y_{j-2}t^{j-2}t^2
=y_0 +y_1t - 2tg_k(t) + 2ty_0 - t^2g_k(t).
\]
Then
\[
g_k(t) = \frac{y_0+2y_0t+ y_1 t}{1+2t+t^2}
=\frac{-y_0-y_1}{(1+t)^2} + \frac {2y_0+y_1}{1+t}
= a\sum_{j=0}^\infty (j+1)(-1)^jt^j + b\sum_{j=0}^\infty (-1)^jt^j
\]
where $a=-y_0-y_1$ and $b=2y_0+y_1$ are linear functions of $y_1,y_2$.  The coefficient of $t^j$, namely $y_j$, is $(-1)^j[b + aj+a]$.  Since $y_j=x_{2^jk}\to0$ it must be that $a=0$ and $b=0$, which means that $y_0=y_1=0$ and therefore $x_k=y_0=0$.  Since $k$ is arbitrary this shows $x=0$ as desired.
\end{proof}
\begin{proof}[Proof of main theorem]
We shall use the generative model result \citep[Theorem 14]{masarotto2018procrustes} that states that if $T_1$ and $T_2$ are self-adjoint, nonnegative definite and have average identity (namely, $T_1+T_2=2\I$), then any covariance $\Sigma$ is a Bures--Wasserstein barycentre of $S_1=T_1\Sigma T_1$ and $S_2 = T_2\Sigma T_2$.

The operators
\[
T_1 = \frac 12F+ \frac12F^* + I,
\qquad T_2 = -\frac 12F - \frac 12F^* + I
\]
are bounded nonnegative self-adjoint and average to the identity so $\Sigma$ is the Fr\'echet mean of $S_1=T_1\Sigma T_1$ and $S_2=T_2\Sigma T_2$.  We have shown that $4S_1$ is injective, and injectivity of $4S_2$ can be shown in the same way, since the polynomial in the denominator of the new generating function is $(1-2t+t^2)$ which admits a root of absolute value $\ge1$.  The proof is thus complete for $n=2$.  It is easy to extend it to any finite $n\ge2$.  It is also possible to take a random variable $a$ with symmetric distribution on $[-1/2,1/2]$ and define a random bounded nonnegative self-adjoint operator $T=a(F+F^*)+\I$.  Then $\mathbb E[T]=\I$, so that $\Sigma$ is the barycentre of the random operator $T\Sigma T$.  If $\mathbb P(a=0)=0$, then $S=T\Sigma T$ is almost surely injective.
\end{proof}

\bibliographystyle{plainnat}
\bibliography{counterInj}

\end{document}